\documentclass[12pt]{amsart}
\usepackage[utf8]{inputenc}
\usepackage{enumerate,enumitem,amsmath,amssymb,mathrsfs,stmaryrd}
\usepackage{latexsym,epsf,graphicx,comment,appendix}
\usepackage{MnSymbol}

\usepackage{tikz}
\usetikzlibrary{arrows.meta}
\usetikzlibrary{automata}
\usetikzlibrary{chains}

\usepackage[english]{babel}
\usepackage{color}
\usepackage{times}
\usepackage[T1]{fontenc}

\usepackage{amsmath}
\usepackage{dsfont}
\usepackage{amsfonts}

\newfont{\bb}{msbm10 at 12pt}
\newfont{\tbb}{msbm10 at 8pt}

\baselineskip 
\usepackage {amsmath}
\usepackage {amsthm}
\usepackage{times}
\usepackage{amscd}
\usepackage{epsf}

\numberwithin{equation} {subsection}

\begin{document}
\mbox{}\vspace{0.2cm}\mbox{}

\providecommand{\keywords}[1]
{
  \small	
  \textbf{\textit{Keywords---}} #1
}

\

\theoremstyle{plain}\newtheorem{lemma}{Lemma}[subsection]
\theoremstyle{plain}\newtheorem{proposition}{Proposition}[subsection]
\theoremstyle{plain}\newtheorem{theorem}{Theorem}[subsection]

\theoremstyle{plain}\newtheorem*{theorem*}{Theorem}
\theoremstyle{plain}\newtheorem*{main theorem}{Main Theorem} 
\theoremstyle{plain}\newtheorem*{lemma*}{Lemma}
\theoremstyle{plain}\newtheorem*{claim}{Claim}

\theoremstyle{plain}\newtheorem{example}{Example}[subsection]
\theoremstyle{plain}\newtheorem{remark}{Remark}[subsection]
\theoremstyle{plain}\newtheorem{corollary}{Corollary}[subsection]
\theoremstyle{plain}\newtheorem*{corollary-A}{Corollary}
\theoremstyle{plain}\newtheorem{definition}{Definition}[subsection]
\theoremstyle{plain}\newtheorem{acknowledge}{Acknowledgment}
\theoremstyle{plain}\newtheorem{conjecture}{Conjecture}

\begin{center}
\rule{15cm}{1.5pt} \vspace{.4cm}

{\bf\Large $p$-Laplace equations in conformal geometry} 
\vskip .3cm

Huajie Liu, Shiguang Ma, Jie Qing, and Shuhui Zhong

\vspace{0.3cm} 
\rule{15cm}{1.5pt}
\end{center}


\title{}

\begin{abstract} In this paper we introduce the $p$-Laplace equations for the intermediate Schouten curvature
in conformal geometry. These $p$-Laplace equations provide more tools for the 
study of geometry and topology of manifolds. First, the positivity of the intermediate Schouten curvature 
yields the vanishing of Betti numbers on locally conformally flat manifolds as consequences of the B\"{o}chner formula 
as in \cite{Nay97, GLW2005}.  Secondly and more interestingly, when the intermediate Schouten curvature
is nonnegative, these $p$-Laplace equations facilitate the geometric applications
of $p$-superharmonic functions and the nonlinear potential theory. This leads to the estimates on Hausdorff dimension of singular sets and 
vanishing of homotopy groups that is inspired by and extends the work in \cite{SY1988, MQ22}. In the forthcoming paper 
\cite{MQ23} we will present our results on the precise asymptotic behavior of $p$-superharmonic functions at singularities.
\end{abstract}



\maketitle

\section{Introduction}\label{Sec:intr}

The potential theory is a classic and powerful approach to the study of partial differential equations.
Recently $n$-Laplace equations as natural extensions of the Gauss curvature equations in conformal geometry 
have been introduced and used for establishing Huber-type theorems in higher dimensions in \cite{MQ21, MQ021}. 
More applications of linear potential theory in conformal geometry have been 
investigated in \cite{MQ22}. Those works renewed and advocated more geometric applications of the  potential theory which are represented by  
geometric applications of subharmonic functions and the potential theory in 2 dimensions, for instance, in \cite{Hu, AH73}. 
In this paper we introduce the $p$-Laplace equations and explore the geometric 
applications of the nonlinear potential theory in conformal geometry for $p\in (2, n)$. We in general assume the dimension of manifolds is
greater than $2$ in this paper unless otherwise specified.


\subsection{$p$-Laplace equations in conformal geometry and applications}
In conformal geometry one often encounters the Schouten curvature tensor on $n$ dimensional manifolds
$$
A = \frac 1{n-2}(Ric - \frac 1{2(n-1)}R\, g)
$$
where $Ric$ stands for the Ricci curvature tensor and $R$ stands for the scalar curvature of the metric $g$ on $n$-manifolds respectively. For 
simplicity, we often let $J =\text{Tr}A =  \frac 1{2(n-1)}R$. In this paper we want to turn the attention to the following intermediate Schouten 
curvature tensor
\vskip 0.05in \noindent
\begin{equation}\label{Equ:p-schouten}
A^{(p)} = (p-2) A + J\, g
\end{equation}
\vskip 0.05in \noindent
for $p\in (1, \infty)$. This curvature tensor plays an intermediate role between the scalar curvature $Jg$ and the Schouten curvature
$A$ for $p\in [2, \infty)$. For the convenience of readers, let us recall (please 
see, for instance, \cite{Lind06})
$$
\Delta_p u = \text{div}(|\nabla u|^{p-2}\nabla u) \text{ and } \Delta_\infty u = u_{ij} u_iu_j.
$$
It turns out that the transformation formula of $A^{(p)}$-curvature under a conformal change gives rise to 
the $p$-Laplace equation
\vskip 0.05in \noindent
\begin{equation}\label{Equ:p-Laplace-intr}
-\Delta_p u + \frac {n-p}{2(p-1)} |\nabla u|^{p-2}A^{(p)}(\nabla u) u = \frac {n-p}{2(p-1)}(|\nabla u|^{p-2}A^{(p)}(\nabla u))[\bar g] u^q
\end{equation}
\vskip 0.05in \noindent
for $\bar g = u^\frac {4(p-1)}{n-p} g$ and $p\in (1, n)$, where $A^{(p)}(\nabla u)$ is $A^{(p)}$-curvature in $\nabla u$ direction and 
$$
q = \frac {2p(p-1)}{n-p}+1\in (1, \infty).
$$
For $p=2$, this is simply the well-known scalar curvature equation
\vskip 0.05in \noindent
\begin{equation}\label{Equ:yamabe-intr}
-\Delta u + \frac {n-2}2 J u = \frac {n-2}2 J [\bar g] u^\frac {n+2}{n-2}
\end{equation}
\vskip 0.05in \noindent
for $\bar g = u^\frac 4{n-2}\, g$. For $p=n$, we recover the $n$-Laplace equations introduced in \cite{MQ21, MQ021}
\vskip 0.05in \noindent
\begin{equation}\label{Equ:n-Laplace-intr}
-\Delta_n \phi + |\nabla \phi|^{n-2}Ric(\nabla \phi) = (|\nabla \phi|^{n-2}Ric(\nabla \phi))[\bar g] e^{n\phi}
\end{equation}
\vskip 0.05in \noindent
for $\bar g = e^{2\phi}\, g$. For $p\in (n, \infty)$, the $p$-Laplace equation \eqref{Equ:p-Laplace-intr} remains valid for 
both $\frac {4(p-1)}{n-p}$ and $q=\frac {2p(p-1)}{n-p} + 1$ being negative. 
In fact, taking $p\to \infty$, it becomes the infinite Laplace equation
\vskip 0.05in \noindent
\begin{equation}\label{Equ:infinite-Laplace-intr}
- \Delta_\infty u - \frac 12 |\nabla u|^2 A(\nabla u) u = -\frac 12 (|\nabla u|^2 A(\nabla u))[\bar g] u^{-7}
\end{equation} 
\vskip 0.05in \noindent
for $\bar g = u^{-4}\, g$. 
\\

We want to mention at this point that there have been interests in the study of quasilinear elliptic equations with the 
nonlinearities similar to those in \eqref{Equ:p-Laplace-intr}, for example, in \cite{BV21, CHZ22}. 
In Section \ref{Sec:intr-p-Laplace} we discuss how the $p$-Laplace equation \eqref{Equ:p-Laplace-intr} 
interpolates the scalar curvature equation \eqref{Equ:yamabe-intr}, $n$-Laplace equations \eqref{Equ:n-Laplace-intr} on Ricci curvature $Ric$, 
and the infinite Laplace equations \eqref{Equ:infinite-Laplace-intr} on the Schouten curvature $A$. We also compare the 
positivity of $A^{(p)}$ with other intermediate positivities of the Schouten curvature tensor $A$. 
\\

As a consequence of Lemma \ref{Lem:cone-comparison} on the comparison of the intermediate positivities of the Schouten
curvature tensor, following the application of the B\"{o}chner formula for harmonic forms in \cite{Nay97, GLW2005}, we observe (see also 
Theorem \ref{Thm:glw-consequence}) 

\begin{theorem}\label{Thm:vanishing-betti-intr} Let $(M^n, g)$ be a compact locally conformally flat manifold with $A^{(p)}\geq 0$ for $p\in [2, n)$. 
Suppose that the scalar curvature is positive somewhere on $M^n$. Then, for $\frac {n-p}2 +1\leq k\leq \frac {n+p}2-1$, the cohomology spaces 
$H^k(M^n, R)=\{0\}$, unless $(M^n, g)$ is isometric to a compact quotient of $\mathbb{H}^r\times \mathbb{S}^{n-r}$ for 
$r = \frac {n-p}2+1\leq \frac n2$.
\end{theorem}

Notice that the situation is well-understood when $p=n$ (see for instance, \cite{MQ021} and references therein). Please check
\eqref{Equ:example} and see how effective and sharp  
the intermediate Schouten curvature $A^{(p)}$ is in detecting the vanishing of cohomology 
and recognizing $\mathbb{H}^k\times\mathbb{S}^{n-k}$. 
\\

More interestingly, we are able to use the Wolff potential to estimate the asymptotic behavior of $p$-superharmonic 
functions at singularities and conclude

\begin{theorem}\label{Thm:main-thm-1-intr} Suppose that $S$ is a compact subset of a bounded domain $\Omega$ in 
$\mathbb{R}^n$.  And suppose that there is a metric $\bar g$ on  $\Omega\setminus S$ that is conformal to the Euclidean 
metric and that $\bar g$ is geodesically complete near $S$. Assume that $A^{(p)} [\bar g]\geq 0$ for some $p\in [2, n)$.
Then
$$
dim_{\mathcal{H}}(S) \leq \frac {n-p}2.
$$
\end{theorem}

Again, when $p=n$, it is known that $S$ can only be at most finitely many points. Please also check again
\eqref{Equ:example} and see how effective and sharp  
the intermediate Schouten curvature $A^{(p)}$ is in estimating the Hausdorff dimensions
and recognizing $\mathbb{H}^k\times\mathbb{S}^{n-k}$. 
There have been extensive studies on the Hausdorff dimensions in \cite{CQY, CH02, CHY, Gn05,Ca12, GMS, GM14, Zh18, DY21, HS21,MQ22} following \cite{SY1988}.  
On a Kleinian manifold $\Omega(\Gamma)/\Gamma$, conversely, a metric is constructed in \cite{Nay97} such that the scalar curvature of this
metric is nonnegative (positive) if the Hausdorff dimension of the set of limit points $L(\Gamma)$ is $\leq \frac {n-2}2$ ($< \frac{n-2}2$). The intermediate Schouten 
curvature tensor $A^{(p)}$ of the Nayatani's metrics in \cite{Nay97} does not seem to be easily understood solely in terms of the Hausdorff dimension $\delta(\Gamma)$
in general. 
\\

However, as a consequence of Theorem \ref{Thm:main-thm-1-intr}, similar to that in \cite{SY1988}, we have (see also Corollary \ref{Cor:vanishing-homotopy}),

\begin{theorem}\label{Thm:main-thm-2-intr} Let $(M^n, g)$ be a compact locally conformally flat manifold with $A^{(p)}\geq 0$ for
$p\in [2, n]$. Then,
for $1< k < \frac {n+p}2 -1$, the homotopy groups $\pi_k(M^n)$ are trivial.
\end{theorem}
 
Theorem \ref{Thm:main-thm-1-intr} and Theorem \ref{Thm:main-thm-2-intr} extends some related results in \cite{SY1988} (and 
for $p=n$ in \cite{MQ21, MQ021}) as the intermediate Schouten curvature and the $p$-Laplace equations substitute the scalar 
curvature and the scalar curvature equations. 

\subsection{The Wolff potentials and $p$-thinness}

In Section \ref{Sec:p-superharmonic-wolff} we review definitions and collect some useful facts in potential theory. And, 
most importantly, we define the concept of $p$-thinness (cf. Definition \ref{Def:quasi-p-thin}) 
 for the study of singular behaviors of $p$-superharmonic functions, which turns out to be different from the thinness 
 defined by the Wiener integral in nonlinear potential theory in general dimensions (cf. \cite{KM94} for instance). The Wolff potential for 
 this paper is defined as follows:
\vskip 0.05in \noindent
 $$
 W^\mu_{1, p} (x, r) = \int_0^r (\frac {\mu(B(x, t))}{t^{n-p}})^\frac 1{p-1} \frac {dt}t
 $$
 \vskip 0.05in \noindent
 for a nonnegative Radon measure $\mu$. It is well known that the Wolff potential is the same as the classic Newton potential when $p=2$.
 We obtain in Section \ref{Sec:p-superharmonic-wolff} that (see also Theorem \ref{Thm:wolff potential upper bdd})

 \begin{theorem}\label{Thm:main-analytic-intr}
 Suppose $\mu$ is a nonnegative finite Radon measure on a bounded domain 
 $\Omega\subset \mathbb{R}^n$. Assume that, for a point $x_{0}\in\Omega$,  some number $m\in (0, n-p)$ and $p\in [2, n)$,
\begin{equation*}
\mu(B(x_{0}, t))\le Ct^{m}
\end{equation*}
for all $t\in (0, 3r_0)$ with $B(x_0, 3r_0)\subset \Omega$. Then, for any $\varepsilon>0$ small, there are a subset $E\subset\Omega$, which
is $p$-thin at $x_{0}$, and a constant $C>0$ such that
\vskip 0.05in \noindent
\begin{equation}\label{Equ:wolff-upper-intr}
W_{1,p}^{\mu}(x, \, r_0)\le C|x - x_0|^{-\frac{n-p-m+\varepsilon}{p-1}} \text{ for all $x\in\Omega\setminus E$.}
\end{equation}
\end{theorem} 
 
 One crucial property of subsets $E$ which are $p$-thin at $x_0$ for us is Theorem \ref{Thm:segment-escape}, which states that
 there is always a ray from $x_0$ that escapes from a small neighborhood of $x_0$ without intersecting $E$. It is remarkable this helps make use of 
 the geodesic completeness in analysis, which seemed first observed in \cite{AH73} for a more straightforward proof of the main theorem on open surfaces 
 in \cite{Hu} (see also \cite{MQ22}).
 
 
 \subsection{Organization}
 The paper is organized as follows: In Section 2 we introduce the $p$-Laplace equations in conformal geometry 
 and discuss the comparisons of the positive cones associated with $A^{(p)}$ to other intermediate positivity of
 the Schouten curvature. In Section 3 we collect definitions and basic relevant facts in the study of 
 $p$-superharmonic functions and the nonlinear potential theory. Most importantly we introduce the concept of $p$-thin 
 in the nonlinear potential theory and prove Theorem \ref{Thm:main-analytic-intr}. In Section 4 we prove Theorem 
 \ref{Thm:main-thm-1-intr} using Theorem \ref{Thm:main-analytic-intr}. In Section \ref{Sec:topology} we derive the topological 
 consequences and prove Theorem \ref{Thm:vanishing-betti-intr} and Theorem \ref{Thm:main-thm-2-intr}. 
 

\section{Introduction of $p$-Laplace equations in conformal geometry}\label{Sec:intr-p-Laplace}

Let us recall the transformation formula of Ricci curvature tensor under a conformal change $\bar g = e^{2\phi} g$
\begin{equation*}
\bar R_{ij} = R_{ij} - \Delta \phi g_{ij} + (2-n) \phi_{i,j} + (n-2) \phi_i\phi_j + (2-n) |\nabla\phi|^2 g_{ij}.
\end{equation*}
For the scalar curvature, we then have
\begin{equation*}
e^{2\phi} \bar R = R - 2(n-1)\Delta \phi  - (n-1)(n-2)|\nabla\phi|^2.
\end{equation*}

\subsection{The intermediate Schouten curvature and its transformation formula}
We write
\begin{equation}\label{Equ:B^p}
A^{(p)}  = \frac 1{n-2}( (p-2)Ric + (n-p)\frac {R}{2(n-1)}\, g),
\end{equation}
where 
$$
A^{(2)} = \frac R{2(n-1)}\, g,  \ A^{(n)} = Ric, \text{ and } \text{Tr}A^{(p)} =\frac {n + p-2}{2(n-1)} \, R.
$$ 
Notice that $\frac {n-p}{n-2} = 1 - \frac {p-2}{n-2}$ and $A^{(p)}$ interpolates between $\frac R{2(n-1)}\, g$ 
and the Ricci curvature tensor $Ric$ as $p$ varies from $2$ to $n$. 
\\

For example, one may calculate the intermediate Schouten
curvature tensor on $\mathbb{H}^k\times\mathbb{S}^{n-k}$ for $1\leq k \leq \frac n2$ and $p= n-2k +2$,
\begin{equation}\label{Equ:example}
\aligned
A^{(p)} & = (p-2)A + Jg = (p-2)(\overset{k}{\overbrace{-\frac 12, \cdots, -\frac 12}}, \overset{n-k}{\overbrace{\frac 12,
\cdots, \frac 12}}) + \frac {n-2k}2 g\\
& =(\overset{k}{\overbrace{0, \cdots, 0}}, \overset{n-k}{\overbrace{n-2k, \cdots, n-2k}})\geq 0.
\endaligned
\end{equation}setminus
Here $\frac{n-p}2 = k-1$ and $\mathbb{H}^k\times \mathbb{S}^{n-k}$ topologically is $\mathbb{S}^n\setminus\mathbb{S}^{k-1}$.
It is also interesting to observe that $A^{(q)}>0$ when $q<n-2k+2$ and $A^{(q)}$ is negative in the first $k$ directions when
$q> n-2k+2$.
\\

For $p\in (1, n)$ and $\bar g = e^{2\phi}\, g = u^\frac {4(p-1)}{n-p} \, g$, we calculate
\begin{equation}\label{Equ:B^p}
\aligned
& A^{(p)}_{ij}[\bar g]  =  A^{(p)}_{ij}  - \frac {2(p-1)}{n-p} \left[ \frac {\Delta u}u g_{ij}  + (p-2) \frac {u_{i,j}}u \right] \, \\
  + \frac {2(p-1)}{n-p} & \left[(1 - (n+p-4)\frac {p-1}{n-p})\frac {|\nabla u|^2}{u^2}g_{ij} + ((p-2)+ 2(p-2)\frac {p-1}{n-p}) \frac {u_iu_j}{u^2} \right].
\endaligned
 \end{equation}
Multiplying $u |\nabla u|^{p-2}\frac {u_i}{|\nabla u|}\frac {u_j}{|\nabla u|}$ to and summing up on both sides of \eqref{Equ:B^p}, 
we arrive at the $p$-Laplace equation
\begin{equation}\label{Equ:p-Laplace-s}
-\Delta_p u + \frac {n-p}{2(p-1)}S^{(p)} (u) u = \frac {n-p}{2(p-1)}S^{(p)} (u)[\bar g] u^q,
\end{equation}
where 
$$
S^{(p)} (u) = |\nabla u|^{p-2} A^{(p)}(\nabla u),
$$
$A^{(p)}(\nabla u)$ is the $A^{(p)}$-curvature in the direction of $\nabla u$, and  
$q = \frac {2p(p-1)}{n-p} + 1 \in (1, \infty)$ monotonically increasing when $p\in (1, n)$. 


\subsection{The $p$-Laplace equations as $p\in (1, \infty)$}

Clearly, when $p=2$, \eqref{Equ:p-Laplace-s} becomes the scalar
curvature equation
$$
-\Delta u +\frac {n-2}{4(n-1)}R u = \frac {n-2}{4(n-1)}R[u^\frac 4{n-2}\, g]u^\frac {n+2}{n-2}.
$$
On the other hand, as $p\to n$, we calculate
$$
u -1 = e^{\frac {n-p}{2(p-1)}\phi} - 1 = \frac {n-p}{2(n-1)} \phi  + o(n-p),
$$
$$
u^q = e^{\frac {(n-p)q}{2(n-1)}\phi}\to e^{n\phi},
$$
$$
\nabla u = \frac {n-p}{2(n-1)}\nabla \phi\,  e^{\frac {n-p}{2(n-1)}\phi},
$$
$$
(\frac {2(p-1)}{n-p})^{p-1}\Delta_p u \to \Delta_n \phi, \text{ and } (\frac {2(n-1)}{n-p})^{p-2}S^{(p)}(u) \to |\nabla \phi|^{p-2}Ric(\nabla \phi).
$$
Therefore, from \eqref{Equ:p-Laplace-s}, as $p\to n$, we recover the $n$-Laplace equation 
$$
-\Delta_n \phi + |\nabla\phi|^{n-2}Ric(\nabla\phi) = (|\nabla\phi|^{n-2}Ric(\nabla\phi))[e^{2\phi}\, g] e^{n\phi},
$$ 
which was introduced and studied in \cite{MQ21, MQ021}. 
Thus $p$-Laplace equation \eqref{Equ:p-Laplace-s} in some way interpolates between the scalar curvature equation and the $n$-Laplace
equation on Ricci curvature for $p\in [2, n]$ in conformal geometry.
\\

The $p$-Laplace equations  \eqref{Equ:p-Laplace-s} still make sense for $p\in (n, \infty)$. As we mentioned in the introduction section, we may write 
$$A^{(p)} = (p-2) A + J\, g$$ 
and therefore $\lim_{p\to\infty} \frac 1{p-2} A^{(p)}  = A$.  Also notice that
$$
\lim_{p\to \infty} \frac 1{p-2} \frac 1{|\nabla u|^{p-4}} \Delta_p u = \Delta_\infty u = u_{ij}u_iu_j
$$
(cf. \cite[Chapter 8]{Lind06}). Therefore, from \eqref{Equ:p-Laplace-s}, after dividing both sides by $(p-2)|\nabla u|^{p-4}$ and taking limit as $p\to\infty$, we get
the infinite Laplace equation on Schouten curvature
\begin{equation}\label{Equ:infinite-Laplace}
- \Delta_\infty u - \frac 12 |\nabla u|^2 A(\nabla u) u = -\frac 12 (|\nabla u|^2 A(\nabla u))[\bar g] u^{-7},
\end{equation}
where $\bar g = u^{-4}\, g$. Therefore $p$-Laplace equation \eqref{Equ:p-Laplace-s} also interpolates between $n$-Laplace equation on Ricci
curvature and the infinite Laplace equation on Schouten curvature for $p\in (n, \infty)$ in conformal geometry.


\subsection{Intermediate positivity of the Schouten curvature tensor}\label{Subsec:cone-comparison}

By the discussion about the interpolation nature of the $p$-Laplace equations \eqref{Equ:p-Laplace-s} 
in the previous subsection, the positivity of $A^{(p)}$ describes stronger and stronger positivity of the Schouten 
curvature $A$ as $p$ increases in $[2, \infty)$. A way to measure the positivity of $A^{(p)}$ is to consider the cone
\begin{equation}\label{Equ:p-cone}
\mathcal{A}^{(p)} = \{(\lambda_1, \lambda_2, \cdots, \lambda_n)\in \mathbb{R}^n: (p-2)\min\{\lambda_i\}_{i=1}^n + \sum_{i=1}^n \lambda_i \geq 0\}.
\end{equation} 
Obviously, we have

\begin{lemma}\label{Lem:theta-cone} \quad
\begin{itemize}
\item  $\mathcal{A}^{(p_2)}\subset \mathcal{A}^{(p_1)} \text{ for $p_1 < p_2$}$;
\item $\mathcal{A}^{(2)} = \{(\lambda_1, \lambda_2, \cdots, \lambda_n)\in \mathbb{R}^n: \sum_{i=1}^n\lambda_i \geq 0\}$ is the baseline;
\item $\mathcal{A}^{(n)}$ stands for $Ric \geq 0$ when $(\lambda_1, \lambda_2, \cdots, \lambda_n)$ represents the curvature tensor $A$;
\item $\mathcal{A}^{(p)}$ approaches the nonnegative cone $\Gamma^n=\{\lambda\in \mathbb{R}^n:\lambda_i\geq 0\}$ as $p\to\infty$.
\end{itemize}
\end{lemma}

It is interesting to compare the intermediate positivity cones $\mathcal{A}^{(p)}$ with those, for example, in connection to the curvature 
operator in the B\"{o}chner formula for harmonic forms \eqref{Equ:bochner} on locally conformally flat manifolds. 
In \cite{Nay97, GLW2005}, for the positivity of the curvature operator $\mathcal{R}$ on $r$-forms in \eqref{Equ:bochner}, 
the following symmetric cones are considered:
\begin{equation}\label{Equ:GLW-cone}
\mathcal{R}^{(r)} = \{(\lambda_1, \lambda_2, \cdots, \lambda_n)\in \mathbb{R}^n: 
\min \{ (n-r)\sum_{k=1}^r\lambda_{i_k} + r \sum_{k=r+1}^n \lambda_{i_k}\} \geq 0\}
\end{equation} 
for $r\in\{1, 2, \cdots, n\}$, where $(i_1, i_2, \cdots, i_n)$ is any permutation of $(1, 2, \cdots, n)$ and the minimum is taken over all permutations.
We observe the following

\begin{lemma}\label{Lem:cone-comparison}
\begin{equation}\label{Equ:g-cone}
\mathcal{R}^{(s)} \subset \mathcal{R}^{(r)} \quad \text{ for $0 < s \leq r \leq \frac n2$}\\
\end{equation}
and
\begin{equation}\label{Equ:cone-comparison}
\mathcal{A}^{(p)}  \subset \mathcal{R}^{(r)} \quad \text{ for $\frac {n-p}2 + 1\leq r \leq \frac n2$}.
\end{equation}
\end{lemma}
\begin{proof} Assume without loss of generality that $\lambda_1\leq \lambda_2\leq\cdots\leq \lambda_n$. And let
$$
G(n, s) = (n-s)\sum_{i=1}^s\lambda_i + s\sum_{i=s+1}^n \lambda_i.
$$
We calculate, for $0 < s < r \leq \frac n2$, 
\begin{equation}\label{Equ:comparison-g-r}
\aligned
G(n, r) & = (n-r)\sum_{i=1}^r \lambda_i +r\sum_{i=r+1}^n\lambda_i\\
 & =  \frac {n-r}{n-s} G(n, s)  + \frac {n(r-s)}{n-s}\sum_{i=s+1}^n \lambda_i + (n-2r)\sum_{i=s+1}^r  \lambda_i\\
&\geq \frac {n-r}{n-s} G(n, s) +\frac {n(r-s)}{n-s}\sum_{i=s+1}^n \lambda_i + \frac {(n-2r)(r-s)}s \sum_{i=1}^s \lambda_i\\
&= \frac {n-r}{n-s} G(n, s)   + \frac {2r(r-s)}{n-s} \sum_{i=s+1}^n \lambda_i + \frac {(n-2r)(r-s)}{s(n-s)}G(n, s)\\
& \geq \frac {n-r}{n-s} G(n, s)   + \frac {2r(r-s)}n \sum_{i=1}^n \lambda_i + \frac {(n-2r)(r-s)}{s(n-s)}G(n, s).
\endaligned
\end{equation}
Hence we may conclude that $\mathcal{R}^{(r)}\subset \mathcal{R}^{(s)}$ when $0<s<r\leq \frac n2$. To verify \eqref{Equ:cone-comparison},
we may assume the smallest eigenvalue $\lambda_1 < 0$ for the sake of the argument, because otherwise $G(n, r)\geq 0$. We calculate
\begin{equation}\label{Equ:r-index}
\aligned
\frac 1r G(n, r) & = (n- 2r)\frac 1r \sum_{i=1}^r \lambda_i + \sum_{i=1}^n\lambda_i\\
\geq (n - 2r) \lambda_1 & +  \sum_{i=1}^n\lambda_i \geq (p-2)\lambda_1 + \sum_{i=1}^n\lambda_i
\endaligned
\end{equation}
if let $n - 2r  \leq p-2$ or equivalently $\frac {n-p}2 + 1\leq r \leq \frac n2$. This completes the proof of this lemma.
\end{proof}

We will see in Section \ref{Sec:topology} that \eqref{Equ:cone-comparison} has some topological
consequences from the B\"{o}chner formula for harmonic forms as observed in \cite{Nay97, GLW2005},
which further confirms that, more positive the curvature is, more vanishing theorems on the cohomology there are. 


\section{$p$-superharmonic functions and the Wolff potentials}\label{Sec:p-superharmonic-wolff}

In this section, we start with definitions and then collect some relevant facts in the potential theory for the study of $p$-superharmonic functions.
We also state and prove the key upper bound estimates of the Wolff potentials (cf. Theorem \ref{Thm:wolff potential upper bdd}) 
for our uses in the following sections.  


\subsection{$p$-superharmonic functions and the Wolff potentials}\label{Subsec:intr-p-wolff}

First we want to give definitions of $p$-harmonic and  $p$-superharmonic functions. In this section we always assume $1<p<n$ unless specified otherwise.
Let us recall again what is the $p$-Laplace operator
\begin{equation*}
\Delta_p u = \text{div}(|\nabla u|^{p-2}\nabla u).
\end{equation*}

\begin{definition}\label{Def:p-harmonic} (\cite[Definition 2.5]{Lind06})
We say that $u\in W^{1,p}_{\rm{loc}}(\Omega)$ is a weak solution of the $p$-harmonic equation in $\Omega$, if 
\begin{equation*}
\int \langle |\nabla u|^{p-2}\nabla u,\nabla \eta \rangle dx=0
\end{equation*}
for each $\eta\in C_0 ^{\infty}(\Omega)$. If, in addition, $u$ is continuous, we then say $u$ is a $p$-harmonic function.
\end{definition}

\begin{definition}\label{Def:p-superhar} (\cite[Definiton 5.1]{Lind06})
A function $v:\Omega\to (-\infty,\infty]$ is called $p$-superharmonic in $\Omega$, if 
\begin{itemize}
\item $v$ is lower semi-continuous in $\Omega$;
\item $v\not\equiv\infty$ in $\Omega$;
\item for each domain $D\subset\subset\Omega$ the comparison principle holds, that is, 
if $h\in C(\bar{D})$ is $p$-harmonic in $D$ and $h|_{\partial D}\le v|_{\partial D}$, then $h\le v$ in D. 
\end{itemize}
\end{definition}

As stated in \cite[Theorem 2.1]{KM92}, for $u$ to be a $p$-superharmonic function in $\Omega$, there is a nonnegative finite Radon 
measure $\mu$ in $\Omega$ such that

\begin{equation}\label{Equ:p-laplace-measure}
-\Delta_p u=\mu.
\end{equation}

For more about $p$-superharmonic functions and nonlinear potential theory, we refer to  \cite{KM94, HKM, HeKi, PV08, AH99, Lind06} and
references therein. The most important tool is the following Wolff potential to substitute the role of Riesz potential in linear potential theory

\begin{equation}\label{Equ:wolff-potential}
W^{\mu}_{1,p}(x,\, r)=\int_0^r(\frac {\mu(B(x,\, t))}{t^{n-p}})^{\frac 1 {p-1}}\frac {dt}{t}
\end{equation}
for any nonnegative Radon measure $\mu$ and $p\in (1, n]$. 
The fundamental estimate for the use of the Wolff potential in the study of $p$-superharmonic
functions is as follows:

\begin{theorem}\label{Thm:main-use-wolff} (\cite[Theorem 1.6]{KM94}) 
Suppose that $u$  is a nonnegative $p$-superharmonic function satisfying \eqref{Equ:p-laplace-measure} for a nonnegative 
finite Radon measure $\mu$ in $B(x, 3r)$. Then
\begin{equation}\label{Equ: main-use-wolff}
c_1 W^{\mu}_{1,p}(x, r)\le u(x)\le c_2(\inf_{B(x, r)}u+W^{\mu}_{1,p}(x, 2r))
\end{equation}
for some constants $c_1(n,p)$ and $c_2(n,p)$ for $p\in (1, n]$.
\end{theorem}

It is often useful to have the following existence result in order to fully use the above estimates (cf. \cite[Theorem 1.6]{KM94}).

\begin{lemma}\label{Lem:existence-p-sh} (\cite[Theorem 2.4]{KM92})
Suppose $\Omega\subset\mathbb{R}^n$ is a bounded domain and $\mu$ is a nonnegative finite Radon measure. 
There is a nonnegative $p$-superharmonic function $u(x)$ satisfying \eqref{Equ:p-laplace-measure}
and $\min\{u,\lambda\} \in W^{1,p}_0(\Omega)$ for any $\lambda >0$.
\end{lemma}


\subsection{$p$-capacity and $p$-thinness}
The $p$-Laplace equation \eqref{Equ:p-laplace-measure} is naturally related to the energy functional $\int_\Omega |\nabla u|^p dx$.
Therefore the following $p$-capacity in terms of the ``energy" defined, for example, in \cite[Chapters 2 and 4]{HKM} and \cite[Section 3.1]{KM94}, 
replaces the capacity in terms of the ``charges" for studying Riesz potentials.

\begin{definition}\label{Def:p-capacity} (\cite[Section 3.1]{KM94}) 
For a compact subset $K$ of a domain $\Omega$ in $\mathbb{R}^n$, we define
\begin{equation}\label{Equ:p-capacity-K}
cap_p(K,\Omega)=\inf \{\int_{\Omega}|\nabla u|^p dx: \text{ $u\in C_0^{\infty}(\Omega)$ and $u\ge 1$ on $K$}\}.
\end{equation}
Then $p$-capacity for arbitrary subset $E$ of $\Omega$ is 
\begin{equation}\label{Equ:p-capacity}
cap_p(E,\Omega)=\inf_{
\text{open}\ G \supset E \text{ \& } G  \subset\Omega} \ \ \sup_{\text{compact} \ K \subset G} cap_p(K,\Omega).
\end{equation}
\end{definition}

We would like to mention some basic facts about the capacity $cap_p(E, \Omega)$ which are important for the arguments in this paper. Readers
are referred to \cite{HKM, KM94}.

\begin{lemma}\label{Lem:subadditivity-Lip}\quad

\begin{enumerate}
\item Countable sub-additivity: 
$$
cap_p(\bigcup_k E_k, \Omega) \leq \sum_k cap_p(E_k, \Omega)
$$
for a countable collection of subsets $\{E_k\}$ of $\Omega$.
\item Scaling property:
$$
cap_p(E_\lambda, \Omega_\lambda) = \lambda^{n-p} cap_p(E, \Omega)
$$
where $A_\lambda = \{\lambda x\in \mathbb{R}^n: x\in A\}$ is the scaled set of $A \subset\mathbb{R}^n$.
\item Lipschitz property: Let 
$$
\Phi:\mathbb{R}^n\to \mathbb{R}^n
$$
be a Lipschitz map, that is, there is a constant $C_0>0$ such that 
$$
|\Phi(x)-\Phi(y)| \leq C_0|x-y| \text{ for all $x, y\in \mathbb{R}^n$}.
$$
Then, there is a constant $C>0$ such that
$$
cap_p (\Phi(E), \mathbb{R}^n) \leq C cap_p(E, \mathbb{R}^n).
$$
\item Positivity:
$$
cap_p(\partial B(0, r), B(0, 2r)) = c(n, p) r^{n-p} >0.
$$
\end{enumerate}
\end{lemma} 
\begin{proof} The proof and references are here: (1) is from \cite[Theorem 2.2]{HKM}; (2) is easily verified; 
(3) is easily been verified based on Definition \ref{Def:p-capacity} 
(see also \cite[Theorem 5.2.1]{AH99}); (4) has been explicitly computed, for instance, \cite[Eample 2.12]{HKM}.
\end{proof}

The following is also very important to us.

\begin{lemma}\label{Lem:cap-estimate} (\cite[Lemma 3.9]{KM94}) 
Suppose that $u\in W_0^{1,p}(\Omega)$ is $p$-superharmonic functions. Then for $\lambda>0$ it holds
\begin{equation}\label{Equ:cap-estimate}
\lambda^{p-1}cap_p(\{x\in\Omega:u(x)>\lambda\},\Omega)\le\mu(\Omega).
\end{equation}
\end{lemma}

The notions of thinness in the potential theory are vitally important. 
The readers are referred to \cite{Br40, Br44, AM72, AH73, HeKi,KM94, MQ21} for detailed discussions and references therein. 
The discussion in \cite[Section 2.3]{MQ21} for the two different notions of thinness, which turns out to be equivalent in dimension 2 due to 
\cite[Theorem 2]{Br44}, is relevant. To deal with the singular behavior of $p$-superharmonic functions, like \cite[Definition 3.1]{MQ21} (see also
\cite[Section 2.5]{Mi96}), we 
propose a thinness that is less restrictive than that by the Wiener type integral when $p\in (2, n)$.
Let 
\begin{eqnarray*}
\omega_i(x_0)&=&\{x\in\mathbb R^n:2^{-i-1}\le|x-x_0|\le2^{-i}\};\\
\Omega_i(x_0)&=&\{x\in\mathbb R^n:2^{-i-2}\le|x-x_0|\le2^{-i+1}\}.
\end{eqnarray*}

\begin{definition}\label{Def:quasi-p-thin}
A set $E\subset\mathbb R^n$ is said to be $p$-thin for $p\in (1, n)$ at $x_0\in\mathbb R^n$ if 
\begin{equation}\label{Equ:quasi-p-thin}
\sum_{i=1}^\infty \frac{cap_p(E\cap\omega_i(x_0),\Omega_i(x_0))}{cap_p(\partial B(x_0, 2^{-i}), B(x_0, 2^{-i+1}))}  <+\infty.
\end{equation}
Meanwhile a set $E$ is said to be $n$-thin at $x_0\in \mathbb{R}^n$ if
\begin{equation}\label{Equ:quasi-n-thin}
\sum_{i=1}^\infty i^{n-1}\, cap_n(E\cap\omega_i (x_0), \Omega_i(x_0)) < +\infty.
\end{equation} 
\end{definition}

To compare, for the thinness used in earlier works, for instance, \cite[Theorem 1.3]{KM94}, which has been 
confirmed to be equivalent to the Cartan property (cf. \cite[(1.4)]{KM94}), we recall that

\begin{equation*}
W_p(E,x_0)=\int_0^1(\frac{cap_p(B(x_0,t)\cap E,B(x_0,2t))}{cap_p(B(x_0,t),B(x_0,2t))})^{\frac{1}{p-1}}\frac{dt}{t}<+\infty.
\end{equation*}
or equivalently (cf. \cite[Lemma 12.10]{HKM})
\begin{equation}\label{Equ:old-p-thin}
\sum_{i=1}^\infty(\frac{cap_p(B(x_0, 2^{-i})\cap E, B(x_0, 2^{-i+1}))}{cap_p(B(x_0, 2^{-i}), B(x_0, 2^{-i+1}))})^\frac 1 {p-1}<+\infty.
\end{equation}
Clearly, a set $E\subset\Omega\subset\mathbb{R}^n$ is $p$-thin at $x_0\in\Omega$ by Definition \ref{Def:quasi-p-thin}
if it is $p$-thin at $x_0$ by \eqref{Equ:old-p-thin} when $p\in (2, n)$ (cf. \cite[Remark 5.1 in Chapter 2]{Mi96}). 
The most important fact about $p$-thin to us in this paper is the following (cf. \cite[Theorem 2.1]{MQ22})

\begin{theorem}\label{Thm:segment-escape} 
Let $E$ be a subset in the Euclidean space $R^n$ and $x_0\in R^n$ be a point. Suppose that $E$ is $p$-thin at the 
point $x_0$ for $p\in (1, n]$. Then there is a ray from $x_0$ that avoids $E$ at least within some small ball at $x_0$.
\end{theorem}

\begin{proof} In the light of properties in Lemma \ref{Lem:subadditivity-Lip} the proof goes exactly as the proof of \cite[Theorem 2.1]{MQ22}.
Applying the rescaling from $\Omega_i$ to $\Omega_{-2}$ and the projection from $\Omega_{-2}$ to the unit sphere, it is then easily seen 
that the $p$-capacity of the union of all the projections of rescaled $\{E\cap \omega_i\}_{i=k}^\infty$ is smaller than the capacity of the unit sphere in 
$\Omega_0$ when $k$ is sufficiently large, which guarantees that there are rays from $x_0$ that avoids $E$ in some small ball. 
\end{proof}


\subsection{The key estimates of the Wolff potential}
In this subsection, we want to state and prove the following estimates on the Wolff potentials, which is essential for our 
investigation of singularities of $p$-superharmonic functions in Section \ref{Sec:Hausdorff-dimension}.

\begin{theorem}\label{Thm:wolff potential upper bdd}
Suppose $\mu$ is a nonnegative finite Radon measure in $\Omega$. Assume that, for a point $x_{0}\in\Omega$ and some number $m\in (0, n-p)$,
\begin{equation}\label{Equ:lebesgue-p}
\mu(B(x_{0}, t))\le Ct^{m}
\end{equation}
for all $t\in (0, 3r_0)$ with $B(x_0, 3r_0)\subset \Omega$. Then, for $\varepsilon>0$, there are a subset $E\subset\Omega$, which
is $p$-thin at $x_{0}$, and a constant $C>0$ such that
\begin{equation}\label{Equ:wolff-upper}
W_{1,p}^{\mu}(x, \, r_0)\le C|x - x_0|^{-\frac{n-p-m +\varepsilon}{p-1}} \text{ for all $x\in\Omega\setminus E$}
\end{equation}
for $p\in [2, n)$.
\end{theorem}

\begin{proof} The proof of this theorem is in principle similar to the proof of \cite[Theorem 2.2]{MQ22}. For $x\in \omega_i$, let us write
\begin{equation}\label{Equ:4-terms}
\aligned
W_{1,p}^{\mu}& (x,r_{0}) =\int_{0}^{r_{0}}(\frac{\mu(B(x,t))}{t^{n-p}})^{\frac{1}{p-1}}\frac{dt}{t}\\
 & = \left\{\int_{2^{-i_0}}^{r_0} + \int_{2^{-i-1}}^{2^{-i_0}} +\int_{2^{-i-2}}^{2^{-i-1}} + \int_{0}^{2^{-i-1}}  \right\}
 (\frac{\mu(B(x,t))}{t^{n-p}})^{\frac{1}{p-1}}\frac{dt}{t}.
\endaligned
\end{equation}
for some large $i_0$ to be fixed by $r_0$. For the first term in the right side of \eqref{Equ:4-terms}
\begin{align*}
\int_{2^{-i_0}}^{r_0} & (\frac{\mu(B(x,t))}{t^{n-p}})^{\frac{1}{p-1}}\frac{dt}{t} \leq \frac {p-1}{n-p}\mu(\Omega)^\frac 1{p-1} (2^{i_0})^\frac {n-p}{p-1}\\
& \leq C(n, p, \varepsilon, i_0) |x - x_0|^{ -\frac{n-p -m +\varepsilon}{p-1}}.
\end{align*}
For the second term in the right side of \eqref{Equ:4-terms}
\begin{align*}
\int_{2^{-i-1}}^{2^{-i_0}} &   (\frac{\mu(B(x,t))}{t^{n-p}})^{\frac{1}{p-1}} \frac{dt}{t} 
 =   \sum_{k=-i -1}^{-i_0 - 1} 
\int_{2^k}^{2^{k+1}} (\frac{\mu(B(x,t))}{t^{n-p}})^{\frac{1}{p-1}}\frac{dt}{t} \\
\leq & \sum_{k= -i-1}^{-i_0} \frac{p-1}{n-p} (\mu(B(x_0, 2^{k+2}))^\frac 1{p-1} (2^k)^\frac {n-p}{p-1}\\
\leq & c(n, p) (\sum_{l=0}^{i -i_0} ((2^\frac{n-p-m}{p-1})^{-l})|x-x_0|^{-\frac {n-p-m}{p-1}} \\
\leq & C(n, p, m, \varepsilon)  |x-x_0|^{-\frac {n-p-m +\varepsilon}{p-1}}.
\end{align*}
For the third term in the right side of \eqref{Equ:4-terms}
\begin{align*}
\int_{2^{-i-2}}^{2^{-i-1}}  & (\frac{\mu(B(x,t))}{t^{n-p}})^{\frac{1}{p-1}}\frac{dt}{t} \leq 
\frac {p-1}{n-p} \mu(B(x_0, 2^{-i+1}))^\frac 1{p-1} (2^{-i-2})^{-\frac {n-p}{p-1}}\\
& \leq C(n, p, \varepsilon) |x-x_0|^{- \frac {n-p-m+\varepsilon}{p-1}}.
\end{align*}
For the fourth term in the right side of \eqref{Equ:4-terms}, we now switch the gear and estimate $cap_p(E_i, \Omega_i)$, where
$$
E_i = \{x\in \omega_i: W_{1, p}^\mu(x, 2^{-i-2}) \geq C |x-x_0|^{-\frac{n-p-m+\varepsilon}{p-1}}\}.
$$
For this purpose, we let $u_i$ be the nonnegative $p$-superharmonic function that solves 
$$
\left\{\aligned -\Delta_p u_i & = \mu \text{ in $\Omega_i$}\\
u_i & = 0 \text{ on $\partial \Omega_i$}. \endaligned\right.
$$
The existence of such $u_i$ is given by Lemma \ref{Lem:existence-p-sh} 
(cf. \cite[Theorem 2.4]{KM92}). Applying Theorem \ref{Thm:main-use-wolff} (cf. \cite[Theorem 1.6]{KM94}), we have
$$
c(n, p)W_{1, p}^\mu (x, 2^{-i-2}) \leq u_i(x) \text{ for $x\in \omega_i$.}
$$
Therefore, in the light of Lemma \ref{Lem:cap-estimate} (cf. \cite[Lemma 3.9]{KM94}), we have
$$
cap_p(E_i, \Omega_i) \leq c(n, p) \mu(\Omega_i) |x-x_0|^{n-p-m+\varepsilon} \leq c(n, p) 2^{-(n-p)i} 2^{-\varepsilon i}.
$$ 
Thus, if let $E = \bigcup_i E_i$, we have
$$
\sum_i 2^{(n-p)i} cap_p (E\bigcap\omega_i, \Omega_i) \leq c(n, p) \sum_i 2^{-\varepsilon i} < \infty,
$$
which implies that $E$ is $p$-thin according to Definition \ref{Def:quasi-p-thin}. Here we have used the facts 
that are stated as (2) and (4) in Lemma \ref{Lem:subadditivity-Lip}. So the proof is completed.
\end{proof}


\section{Applications in conformal geometry}\label{Sec:Hausdorff-dimension}

In this section we will use the $p$-Laplace equation \eqref{Equ:p-Laplace-s} and Theorem \ref{Thm:wolff potential upper bdd} to derive the
consequence of the curvature condition $A^{(p)}\geq 0$ and prove Theorem \ref{Thm:main-thm-1-intr}. Let us first restate it as follows:

\begin{theorem}\label{Thm:main-geometric-thm-1}
Suppose that $S$ is a compact subset of a bounded domain $\Omega\subset\mathbb{R}^n$.  And suppose that there is a metric $\bar g$ on 
$\Omega\setminus S$ such that 
\begin{itemize}
\item it is conformal to the Euclidean metric $g_{\mathbb{E}}$;
\item it is geodesically complete near $S$. 
\end{itemize}
Assume that $A^{(p)} [\bar g]\geq 0$ for some $p\in [2, n)$.
Then
$$
dim_{\mathcal{H}}(S) \leq \frac {n-p}2
$$
for $p\in [2, n)$.
\end{theorem}

\begin{proof} Let $\bar g = u^\frac {4(p-1)}{n-p}\, g_{\mathbb{E}}$ for $u>0$ in $\Omega\setminus S$. 
First, in the light of \eqref{Equ:p-Laplace-s} and the assumption on the curvature
$A^{(p)}[\bar g]$, we know
\begin{equation}\label{Equ:geometry}
-\Delta_p u = \frac {n-p}{2(p-1)} |\nabla u|^{p-2}A^{(p)}(\nabla u) u^q = f \geq 0 \text{ in $\Omega\setminus S$.}
\end{equation}
{\bf Step 1}: We claim that $u$ is actually a $p$-superharmonic function such that
\begin{equation}\label{Equ:claim}
-\Delta_p u = \mu \text{ in $\Omega$}
\end{equation}
for some nonnegative finite Radon measure $\mu$ on $\Omega$. Due to \cite[Lemma 3.1]{MQ021} (see also \cite[Proposition 8.1]{CHY}), 
we know
$$
u(x)\to \infty \text{ as $x\to S$}
$$
since the scalar curvature of the conformal metric $\bar g$ is nonnegative from $A^{(p)}[\bar g]\geq 0$. In fact, from \cite[Theorem 3.1]{MQ22},
we also know 
$$
dim_{\mathcal{H}} (S)\leq \frac {n-2}2.
$$
To derive \eqref{Equ:claim} for some nonnegative finite Radon measure $\mu$,  
as in the proof of \cite[Lemma 3.2]{MQ22} (see also \cite{DHM, MQ021}), we use the cut-off function based on the level sets of $u$. Let
$$
\alpha_s (t) = \left\{ \aligned t \quad & t\in [0, s];\\
\text{increasing} \quad & t\in [s, 10s];\\
2s \quad & t\in [10s, \infty) \endaligned\right.
$$
be a function satisfying $\alpha_s' = \frac d{dt}\alpha_s \in [0, 1]$ and $\alpha_s'' = \frac {d^2}{dt^2}\alpha_s \leq 0$. 
The proof from here goes exactly the same line by line of the proof of \cite[Lemma 3.2]{MQ22} with little changes. And the non-negativity of the 
finite Radon measure $(-\Delta_p u)$ is readily verified due to \eqref{Equ:geometry}. 
\\

\noindent
{\bf Step 2:} Next we apply Theorem \ref{Thm:wolff potential upper bdd} to derive the Hausdorff dimension estimate. The approach runs 
similarly to the proof of \cite[Theorem 3.1]{MQ22}. Assume otherwise 
$$
dim_{\mathcal{H}}(S) > \frac {n-p}2.
$$
Therefore, there is $\varepsilon > 0$ such that $\mathcal{H}_{\frac {n-p + 3\varepsilon}2}(S) = \infty$. Applying \cite[Proposition 1.4]{Kp19}
(see also \cite[Lemma 2.7]{MQ22}), there is a point $x_0\in S$ such that
$$
\mu(B(x_0, t)) \leq Ct^{\frac {n-p+ 3\varepsilon}2}
$$
for all $t\in (0, r_0)$ and $B(x_0, 3r_0)\subset\Omega$. In the light of Theorem \ref{Thm:wolff potential upper bdd}, we have
\begin{equation}\label{Equ:application}
W_{1,p}^{\mu}(x, \, r_0)\le C|x - x_0|^{-\frac{n-p -\varepsilon}{2(p-1)}} \text{ for all $x\in\Omega\setminus E$}
\end{equation}
for some $E$ that is $p$-thin at $x_0$. To get the upper bound for the $p$-superharmonic function $u$, we have, from
 Theorem \ref{Thm:main-use-wolff} (see also \cite[Theorem 1.6]{KM94}), we have
$$
u(x)  \leq C (\inf_{B(x, r_0)} u(y) + W^\mu_{1, p}(x, 2r_0))
$$
for $x\in B(x_0, r_0)$. In the light of \eqref{Equ:application}, we have
\begin{equation}\label{Equ:u-upper-bdd}
u(x) \leq C\frac 1{|x-x_0|^\frac {n-p - \varepsilon}{2(p-1)}}
\end{equation}
for $x\in B(x_0, r_0)\setminus (E\cup\{x_0\})$ for a subset $E$ that is $p$-thin at $x_0$.
\\

\noindent
{\bf Step 3:} In this final step we use the geodesic completeness of the conformal metric $\bar g = u^\frac {4(p-1)}{n-p} g_{\mathbb{E}}$
and the fact that there is a ray $\Gamma$ from $x_0$ that avoids $E$ in some small ball around $x_0$ because $E$ is $p$-thin at $x_0$ according
to Theorem \ref{Thm:segment-escape}.  On this ray $\Gamma$, we calculate the length with respect to the conformal metric $\bar g$
$$
\int_0^{l_0} u^\frac {2(p-1)}{n-p} dr \leq C \int_0^{l_0} \frac 1{r^\frac {n-p- \varepsilon}{n-p}} dr < \infty
$$ 
for some $l_0>0$ based on \eqref{Equ:u-upper-bdd}, which contradicts with the geodesic completeness of $\bar g$. Therefore 
$$
dim_{\mathcal{H}}(S) \leq \frac {n-p}2.
$$
Thus the proof is complete.
\end{proof}


\section{Vanishing of Topology }\label{Sec:topology}

In this section we discuss topological consequences for locally conformally flat manifolds to have $A^{(p)}\geq 0$. On homotopy groups, we take the
approach based on the work in Section \ref{Sec:Hausdorff-dimension} and \cite{SY1988}. On cohomology spaces, we derive some consequences 
from the comparison of positive cones with the help from \cite{Nay97, GLW2005}.


\subsection{On homotopy groups} In this subsection we follow the set-up from \cite{SY1988} and use the development map to unfold
a compact locally conformally flat manifold in $\mathbb{S}^n$. We always assume the dimension is greater than $2$. Let us first restate Theorem \ref{Thm:main-geometric-thm-1} for subsets 
in sphere $\mathbb{S}^n$.

\begin{theorem}\label{Thm:main-geometric-thm-2} Suppose that $S$ is a closed subset of the sphere $\mathbb{S}^n$.  
And suppose that there is a metric $\bar g$ on $\mathbb{S}^n\setminus S$ that is conformal to the standard round metric 
$g_{\mathbb{S}}$. Assume that it is geodesically complete near $S$ and that $A^{(p)} [\bar g]\geq 0$ for some $p\in [2, n)$.
Then
$$
dim_{\mathcal{H}}(S) \leq \frac {n-p}2.
$$
\end{theorem}

\begin{proof} Pick a point $x_\infty$ outside $S$ in $\mathbb{S}^n$ and take the stereographic projection from $\mathbb{S}^n\setminus x_\infty$ to 
$\mathbb{R}^n$. Let $\Omega$ be a bounded and open neighborhood of $S$ in $\mathbb{R}^n$. Now this theorem simply follows from 
Theorem \ref{Thm:main-geometric-thm-1}. 
\end{proof}

Next let us consider compact locally conformally flat manifolds $(M^n, g)$ with $A^{(p)}\geq 0$ for $p\in [2, n)$. By \cite{SY1988}, we know
the development map, the conformal immersion from the universal covering $\tilde M^n$ into $\mathbb{S}^n$,

\hskip2.0in\begin{tikzpicture}
    xscale=2.5, yscale 0.5
    \node (x) at (0,0) {$\tilde M^n$};
    \node (y) at (3, 0) {$\mathbb{S}^n$};
    \node (z) at (0, -1.5) {$M^n$};
    \draw[->] (x) edge node[above] {$\Phi$} (y);
    \draw[->] (x) edge node [left]{$\pi$} (z);
\end{tikzpicture}

\noindent
is injective when the scalar curvature $R[g]$ is nonnegative. Then the image $\Omega = \Phi(\tilde M^n)$ in $\mathbb{S}^n$ comes with
the metric $(\Phi^{-1})^*(\tilde g) = \bar g$ that is conformal to the standard round metric $g_{\mathbb{S}}$, where $\tilde g$ is the lift of
$g$ on $\tilde M^n$.
\\

\begin{corollary}\label{Cor:vanishing-homotopy}
Suppose that $(M^n, g)$ is a compact locally conformally flat manifold with $A^{(p)}\geq 0$ for some $p\in (2, n)$. Then
the homotopy groups $\pi_k(M^n)$ vanishes for all $ 1< k < \frac {n +p -2}2$. 
\end{corollary}

\begin{proof} The proof goes exactly like that in \cite{SY1988}. First, by Theorem \ref{Thm:main-geometric-thm-2} in the above, we conclude
$$
dim_{\mathcal{H}}(\mathbb{S}^n\setminus\Omega) \leq \frac {n-p}2.
$$
This is to say that the co-dimension of $\mathbb{S}^n\setminus\Omega$ is at least $\frac {n+p}2$, which implies any map 
$$
h: \mathbb{S}^k\to \Omega\subset\mathbb{S}^n
$$ 
is homotopic to a trivial map without across $\mathbb{S}^n\setminus\Omega$ if $0 < k < \frac {n+p}2 -1$. Therefore the
homotopy groups $\pi_k(\Omega)$ vanishes and therefore the homotopy groups $\pi_k(M^n)$ vanishes 
when $1 < k < \frac {n+p}2-1$ by the property of covering maps. The proof is complete.
\end{proof}

\subsection{On Betti numbers} In this subsection we draw corollaries from the vanishing theorem in \cite{Nay97, GLW2005} and the
comparison of positive cones in Section \ref{Subsec:cone-comparison}. Let us recall from \cite{Nay97, GLW2005} 
the B\"{o}chner formula for harmonic forms
\begin{equation}\label{Equ:bochner}
\Delta \omega = \nabla^*\nabla \omega + \mathcal{R}\omega
\end{equation}
where $\Delta = d^*d + dd^*$ is the Hodge Laplace, $\nabla^*\nabla$ is on the other hand the rough Laplace, and the curvature
operator
$$
\mathcal{R}\omega =  \sum_{i, j=1}^n \omega_i \wedge \iota_{e_j}(R(e_i, e_j)\omega)
$$
where $\{e_k\}$ is a basis, $\{\omega_k\}$ is the dual basis, and $R(e_i, e_j)$ is the Riemann curvature operator on $r$-forms. 
Recall
$$
R_{ijkl} = W_{ijkl} + A_{ik}g_{jl} - A_{il}g_{jk} + A_{jl}g_{ik} - A_{jk}g_{il}.
$$
Particularly, when the metric is locally conformally flat, i.e. $W=0$, one has
$$
R_{ijkl} = A_{ik}g_{jl} - A_{il}g_{jk} + A_{jl}g_{ik} - A_{jk}g_{il}.
$$
So, for an orthonormal basis $\{e_j\}$ under which $A$ is diagonalized with eigenvalues $\{\lambda_j\}$, we have
$$
R_{ijkl} = \lambda_i(\delta_{ik}\delta_{jl} - \delta_{il}\delta_{jk}) + \lambda_j(\delta_{jl}\delta_{ik} - \delta_{jk}\delta_{il})
$$
and therefore 
$$
\mathcal{R} \omega = ((n-r)\sum_{i=1}^r\lambda_i + r \sum_{i=r+1}^n\lambda_i)\omega
$$
for $\omega = \omega_1\wedge\omega_2\cdots\wedge\omega_r$. Thus
$$
<\mathcal{R}\omega, \omega> \geq 0 \text{ when } \lambda(A)\in \mathcal{R}^{(r)}
$$
for any $r$-form $\omega$ according to \eqref{Equ:GLW-cone}. Similar to the proof of 
\cite[Proposition 2.1]{GLW2005} , using Lemma \ref{Lem:cone-comparison} in Section \ref{Subsec:cone-comparison}, we have the following
restatement of Theorem \ref{Thm:vanishing-betti-intr}:

\begin{theorem}\label{Thm:glw-consequence} Suppose that $(M^n, g)$ is a compact locally conformally flat manifold with $A^{(p)}\geq 0$
for some $p\in [2, n)$. And suppose that the scalar curvature is positive at some points. Then, for $ \frac {n-p}2 +1\leq r\leq \frac {n+p}2-1$,
the cohomology spaces $H^r(M^n, R)=\{0\}$, unless $(M^n, g)$ is isometric to a quotient of $\mathbb{H}^r\times \mathbb{S}^{n-r}$ for 
$r = \frac {n-p}2+1 < \frac n2$.
\end{theorem}
\begin{proof} We focus on the cases $p\in (2, n)$. Recall that, if $\mathcal{R}\geq 0$ and $\mathcal{R}>0$ at some point on $r$-forms, then
$H^q(M^n, R)=\{0\}$ for $r\leq q \leq n-r$, which follows from the B\"{o}chner formula \eqref{Equ:bochner}.
In the light of \eqref{Equ:cone-comparison} in Lemma \ref{Lem:cone-comparison}, $\lambda(A)\in \mathcal{R}^{(r)}$ when
$\frac {n-p}2 +1 \leq r\leq \frac n2$ and $A^{(p)}\geq0$. It is also easily seen that $\lambda(A)$ is in the interior of $\mathcal{R}^{(r)}$ if
$\frac1{2(n-1)}R>0$, unless $\frac {n-p}2$ is an integer and the first $r=\frac {n-p}2 +1$ eigenvalues of the Schouten
tensor are the same in the light of \eqref{Equ:comparison-g-r} and \eqref{Equ:r-index}, and in these cases, one has to skip $\frac {n-p}2 +1$
and $\frac {n+p}2 -1$. We claim those exceptional cases one skips are very specific, and in fact, based on the rigidity result in the proof 
of \cite[Proposition 2.1 (iii)]{GLW2005} following \cite{Lich63} (see also \cite{Laf88}), 
they are simply the cases where $(M^n, g)$ is isometric to a quotient of $\mathbb{H}^r\times \mathbb{S}^{n-r}$. 
\end{proof}



\vskip 0.3cm
\noindent Huajie Liu: Department of Mathematics, Nankai University, Tianjin, China; \\e-mail: 
1120220031@nankai.edu.cn 
\vspace{0.2cm}

\noindent Shiguang Ma: Department of Mathematics, Nankai University, Tianjin, China; \\e-mail: 
msgdyx8741@nankai.edu.cn 
\vspace{0.2cm}

\noindent Jie Qing: Department of Mathematics, University of California, Santa Cruz, CA 95064; \\
e-mail: qing@ucsc.edu \\
Research of this author is partially supported by Simons Foundation.
\vspace{0.2cm}

\noindent Shuhui Zhong: School of Mathematics, Tianjin University, Tianjin, China; \\
e-mail: zhshuhui@tju.edu.cn


\end{document}